\documentclass[review]{elsarticle}
\usepackage{amsmath,amssymb,amsthm}
\usepackage{multirow}
\usepackage{lineno,hyperref}
 \usepackage{color}
\usepackage{bm}
\usepackage{braket,amsfonts}

\usepackage{array}

\def\bR{\mathbb{R}}

\newtheorem{theorem}{Theorem}[section]
\newtheorem{lemma}[theorem]{Lemma}

\newtheorem{remark}{Remark}
\newtheorem{definition}[theorem]{Definition}
\modulolinenumbers[5]
\journal{****}

\begin{document}
\begin{frontmatter}
\title{\Large Accelerated proximal iterative hard thresholding method for $\ell_0$ minimization}

\author[a]{Xue Zhang}
\ead{zhangxue2100@sxnu.edu.cn}
\author[a,b]{Xiaoqun Zhang}
\ead{xqzhang@sjtu.edu.cn}

\address[a]{School of Mathematics \& Computer Science, Shanxi Normal University, Shanxi, CHINA}
\address[b]{Institute of Natural Sciences, Shanghai Jiao Tong Univeristy, Shanghai, CHINA}
\begin{abstract}
In this paper,  we consider a non-convex problem which is the sum of $\ell_0$-norm and a convex smooth function under a box constraint. We propose one proximal iterative hard thresholding type method with an extrapolation step for acceleration and  establish its global convergence results. In detail, the sequence generated by the proposed  method globally converges to a local minimizer of the objective function. Finally, we conduct numerical experiments to show the proposed method's effectiveness on comparison with some other efficient methods.
\end{abstract}
\end{frontmatter}

\section{Introduction}
In modern science and technology, signal and image processing problems have many important applications, for example, compressive sensing, machine learning and medical imaging.
Signal and image processing problems can be often formulated as
the following inverse problem
\begin{equation}\label{Ab}
A(x)+\epsilon=b,
\end{equation}
where $A$ is some linear or non-linear operator, $b$ is the observation data, $\epsilon$ is some observation error and $x$ is the vector we wanted. Problem \eqref{Ab} is usually ill-posed, thus solving \eqref{Ab} is non-trivial. To overcome this difficulty, the prior sparsity of the signals or images is usually considered. One often used minimization model is formulated as
\begin{equation}\label{fplusg}
\min_{ x\in X}f(x)+g(x)
\end{equation}
where $f(x)$ is the data fidelity term related to equation \eqref{Ab}, $g(x)$ is some regularization term to promote $x$'s sparsity, and $X\subseteq\bR^n$ is some convex constraint set. A natural idea for sparsity promotion is taking $g(x)=\lambda \|x\|_0$ where $\lambda>0$ is some regularization parameter and the notation $\|x\|_0$, $x$'s $\ell_0$  norm,  denotes the number of $x$'s nonzero elements.

It is well-known that finding the global minimizer of $\ell_0$ regularization problem is NP hard. And it is hard to  develop convergent, efficient and tractable method since  $\ell_0$-
norm is  non-convex and discontinuous.  That is also a reason why the $\ell_1$ convex relaxation model
 \begin{equation}\label{L1}
\min_{x\in X}f(x)+\lambda\|x\|_1
\end{equation}
 are largely adopted.
 However, $\ell_0$ regularization problem still has some advantages over $\ell_1$ regularization problem. For example, $\ell_1$ regularization problem may fail to recover sparse solutions for some very ill-posed
inverse problems and non-Gaussian noise corruption \cite{ZLC12}.
 Compared with $\ell_1$ regularization problem, $\ell_0$ regularization problem can directly recover sparser solutions. Moreover, the continuity of the soft thresholding operator,
\begin{equation}\label{eq:st}
\mathcal{S}_{\lambda}(c)=\arg\min \lambda\|x\|_1+\frac{1}{2}\|x-c\|^2=\mbox{sign}(c)\max\{|c|-\lambda,0\}
\end{equation}
 used for solving $\ell_1$ regularization problem, may yield loss of contrast and eroded signal peaks since all the coefficients are deduced. In statistical learning, it is also well known that $\ell_1$ solution is a biased estimator \cite{Fan2001variable}. In many applications, $\ell_0$ regularization achieves better sparse solution than $\ell_1$ regularization, for example \cite{CCSS03,DZ13,ZDL-2013}. Thus we consider the following $\ell_0$ regularization problem
\begin{equation}\label{fl0}
\min_{x\in X} \lambda\| x\|_0+f(x),
\end{equation}
 and devote to design and discuss an efficient method with simple structure.

 Analogue to the proximal forward-backward splitting (PFBS) method \cite{LM79,Passty1979Ergodic,CW05} for convex problems \eqref{fplusg}, a {proximal iterative hard thresholding} (PIHT) method is used in many works to solve $\ell_0$ regularization problem (\ref{fl0}) when $X=\bR^n$. Its convergence and convergence rate have been studied in \cite{CCSS03,BD08,BD09,HJ,JS,Lu12,Zhang2015A} under different assumptions. Typically, under the assumption that $f(x)$ has Lipschitz continuous gradient, it obtains the next iterative point by  solving a subproblem which contains a
linearization term of $f(x)$ at current iteration point $x^k$ and a proximal term.  In detail, the PIHT method  is given as
\vskip 2cm\hrule\vskip 0.1cm \noindent\textbf {PIHT Algorithm }\vskip 0.1cm\hrule\vskip 0.2cm

\noindent Choose parameters $\mu>0, \lambda>0$, starting point $x_0$; compute the Lipschitz constant $L$ of $\nabla f(x)$; let $k=0$.

\noindent\textbf {while} the stopping criterion does not hold, compute
\begin{equation}
x^{k+1}\in \arg\min_{x \in\bR^n}\lambda \| x\|_0+\frac{L}{2}\|x-x^k+\frac{1}{L}\nabla f(x^k)\|^2+\frac{\mu}{2}\|x-x^k\|^2
\label{eq:iht}
\end{equation}

$k=k+1$

\noindent\textbf {end(while)} \vskip 0.2cm\hrule\vskip 0.5cm
As well known, the step \eqref{eq:iht} can be given by
 $$x^{k+1}\in\mathcal{H}_{\sqrt{\frac{2\lambda}{L+\mu}}} (x^k-\frac{1}{L+\mu}\nabla f(x^k)),$$
 where $\mathcal{H}_{\gamma}(\cdot)$ is the hard thresholding operator, a set-valued componentwise operator, defined as
\begin{equation}\label{eq:ht}
(\mathcal{H}_{\gamma}(c))_i=
\left\{
\begin{array}{ll}
\{c_i\},      &\mbox{if}\; |c_i|>\gamma\\
\{0,c_i\} &\mbox{if}\; |c_i|=\gamma\\
\{0\},        &\mbox{if}\; |c_i|< \gamma
\end{array}
\right.
\end{equation}
where $c_i$ denotes the $i$th component of vector $c$.

Accelerated PFBS methods have been extensively considered for  solving  problem \eqref{fplusg} with convex $f,g$.  For instance, in \cite{BT09,STY-2011,Salzo2012Inexact,Attouch2015The},  extrapolation steps are utilized to  achieve a convergence complexity of  $O(1/k^2)$ (even $o(1/k^2)$ \cite{Attouch2015The}) in terms of objective value error. Similar to the accelerated technique used for accelerated proximal gradient (APG)  method for convex cases, we will propose one accelerated PIHT method for $\ell_0$ minimization using  extrapolation and provide its convergence results.

On solving non-convex problems, many algorithms, such as inertial forward-backward method  \cite{bot2014inertial} (IFB), monotone accelerated proximal gradient method \cite{chubulai} (mAPG), and non-monotone APG method \cite{chubulai} (nmAPG), are proposed to accelerate the convergence of the usual PFBS method. In \cite{BDHSZZ2016}, an extrapolated proximal iterative hard-thresholding (EPIHT) algorithm is proposed to accelerate the PIHT for $\ell_0$ minimization.  The convergence of the above mentioned algorithms are usually build upon Kurdyka-{\L}ojasiewicz (KL) property (for details, one can  see \cite{Xu2013A,Bolte2007Clarke,Bolte2008Characterizations,KURDYKA1994W,Attouch2008Proximal,1963Une})) of objective function. In this paper, we will design an extrapolated proximal algorithm for $\ell_0$ optimization  and tackle the convergence analysis directly without using the tool of KL property. The global convergence to a local minimizer of the proposed algorithm is established purely based on the convexity of $f$ and the property of $\ell_0$ function. Compared to EPIHT and some other algorithms,  one  advantage of our proposed scheme is that  a small amount of function and gradient evaluation are involved at each iteration. The setting of parameters are relatively simple compared to some other related algorithms. Finally, numerical experiments also show the effectiveness of the proposed algorithm. A detail presentation of the related algorithms and comparison will be present in Section  \ref{sec:dis}.

 The rest of the paper is organized as follows. In section \ref{sec:alg}, we introduce the proposed algorithm and establish its convergence results. In section \ref{sec:dis}, we will give a discussion on our method and  the comparison to  other state-of-the-art methods.   In section \ref{sec:test}, we conduct  experiments to show our method's  numerical performance and efficiency.

\section{ Algorithm and its convergence}
\label{sec:alg}
\subsection{Preliminaries}

We first introduce some notations, concepts and results that will be used in this paper.

\begin{itemize}
\item For any $x\in\bR^n$, $x_i$ represents $x$'s $i$-th component.
\item Given any index set $I\subseteq \{1,2,\ldots,n\}$, we let
\[
C_I:=\{x\in\bR^n: x_i=0\text{ for all }i\in I\};
\]
conversely, given any $x\in\bR^n$, we define the zero element index set of a vector $x\in\mathbb{R}^n$ as
\begin{equation}\label{eqn:Ix}
I(x):=\{i:x_i=0\}.
\end{equation}

\item The projection operator defined on a set $C\subseteq \bR^n$ is denoted by
    \[
    P_{C}(x)=\arg\min_{z\in C}\frac{1}{2}\|z-x\|^2.
    \]
    $P_{C}(\cdot)$ is continuous, namely
    \[
    \lim_{k\rightarrow +\infty}P_{C}(x^k)=P_{C}(\lim_{k\rightarrow +\infty}x^k)
     \]
     if $\lim_{k\rightarrow +\infty}x^k$ exists.
 \item For any $x\in \mathbb{R}$, $\|x\|_0=0$ if $x=0$; otherwise $\|x\|_0=1$. Then
 for any positive integer $n$ and $y\in \mathbb{R}^n$, $\|y\|_0=\sum_{i=1}^n\|y_i\|_0$ denotes the number of $y$'s nonzero elements.
\end{itemize}

\smallskip
\begin{definition}
 A mapping $T:\bR^n \rightarrow\bR^n$ is said to be $L_T$-Lipschitz
continuous on the set $X\subseteq\bR^n$ if there exists $L_T>0$ such that
\[
\|T(x)-T(y)\| \leq L_T\|x-y\|,\quad \forall x, y\in X.
\]
\end{definition}

\smallskip
\begin{definition}
Let $f:\bR^n \rightarrow \bR\cup\{+\infty\}$ be a closed proper
convex function, then the subdifferential of $f$ at $x$ is defined by
\[
\partial f(x):=\{s\in\bR^n: f(y)\geq f(x)+\langle s,y-x\rangle,\;
\forall y\in\bR^n\}.
\]
And each element $s\in \partial f(x)$ is called a subgradient of $f$ at point $x$.
Moreover, if $f$ is continuous differentiable,
$\partial f(x)=\{\nabla f(x)\}$.
\end{definition}

\smallskip
\begin{lemma}\cite{BT09}\label{Lip}
$f: \bR^n\rightarrow \bR $ is continuous differentiable. If $\nabla f(x)$ is L-Lipschitz continuous, the following inequality holds
\[
f(x)-f(y)\leq
\langle\nabla f(y),x-y\rangle+\frac{L}{2}\|x-y\|^2, \;\; \forall x,y \in \bR^n.
\]
\end{lemma}

\smallskip
\begin{lemma}\cite{BT09}\label{Lipschitz1}
Denoting
\[
B_{L}(y)=(I+\partial g/L)^{-1}(y-\nabla f(y)/L).
\]
where $g:\bR^n\rightarrow \bR$ is a proper closed convex function, $f:\bR^n\rightarrow \bR$ is convex smooth and $\nabla f$ is $L$-Lipschitz continuous. Letting $h:=f+g$, for any $x,y\in \bR^n$, the following inequality holds
\[
h(x)-h(B_L(y))\geq \frac{L}{2}\|B_{L}(y)-y\|^2+L\langle y-x,B_L(y)-y\rangle.
\]
\end{lemma}

\subsection{Model and algorithm}
In this paper,   we consider the following minimization problem
\begin{equation}\label{eqn:functionH}
\min_{x\in\bR^n} H(x):=\lambda\| x\|_0+f(x)+\delta_X(x),
\end{equation}
where $X=\{x\in \bR^n:l\leq x\leq u \}$ ($l,u$ can be vectors), and the indicator function
\[
\delta_X(x)=
\left\{
\begin{array}{ll}
0, & \text{ if } x\in X;\\
+\infty, & \text{ otherwise}.
\end{array}
\right.
\]

\begin{remark}
If $f(x)$ is coercive, one can take $l=-\infty,u=+\infty$, all the results in this paper  still hold.  If the original problem is unconstrained and $f(x)$ isn't coercive, one can take the elements of $l$ very small and the elements of $u$ very large, for example $l=\{-10^{12}\}^n, u=-l$.
\end{remark}

\begin{remark}
Here we use the uniform parameter $\lambda \|x\|_0$ instead of the weighted $\|\bm{\lambda}\cdot x\|_0:=\sum_{i=1}^n\bm{\lambda}_i\|x_i\|_0$ for the simplicity of notation, while all the results  can be easily extended to the weighted case.
\end{remark}

\smallskip
Throughout this paper, our  assumption on problem \eqref{eqn:functionH} is\\

\newpage
\textbf{Assumption A:}
\begin{enumerate}
  \item $f$ is convex differentiable and  bounded from below on set $X$;
  \item $\nabla f$ is $L$-Lipschitz continuous on set $X$.
\end{enumerate}

\smallskip
For solving problem \eqref{eqn:functionH}, we propose the following \emph{extrapolated} type method.

\vskip 0.5cm\hrule\vskip 0.1cm \noindent\textbf { Algorithm 1}\vskip 0.1cm\hrule\vskip 0.2cm

\noindent Choose parameters $\mu>0, \lambda>0$ and a sequence of extrapolation weights $0<\omega_k\leq \omega<1$; compute the Lipschitz constant $L$ of $\nabla f(x)$; choose starting point $x^{-1}=x^0$;  let $k=0$.

\noindent\textbf {while } the stopping criterion does not hold

Let
\[y^{k+1}_i=
\left\{
\begin{array}{ll}
x^{k}_i,&i\notin I(x^k)\\
x^{k}_i+\omega_k(x^{k}_i-x^{k-1}_i),&i\in I(x^k)
\end{array}
\right.
\]

\textbf{if} $\langle y^{k+1}-x^{k}, \nabla f(y^{k+1})\rangle>0$ \mbox{ or } $y^{k+1}\notin X$
\hspace{1cm}\begin{equation}
y^{k+1}=x^{k}
\label{eq:restart}
\end{equation}

\textbf{end(if)}

\begin{equation}\label{EPIHT}
x^{k+1}\in \arg\min_{x \in X}\lambda\| x\|_0+\frac{L}{2}\|x-y^{k+1}+\frac{1}{L}\nabla f(y^{k+1})\|^2+\frac{\mu}{2}\|x-y^{k+1}\|^2
\end{equation}

$k=k+1$

\noindent\textbf {end(while)} \vskip 0.2cm\hrule\vskip 0.5cm

During the iteration, we assume that the support of  $x^k$ is more accurate than that of  $x^{k-1}$.  The extrapolation is only performed in the subspace $C_{I(x^k)}$. The gradient information is used to determine whether the extrapolation step will be accepted. In fact, if $\langle y^{k+1}-x^{k}, \nabla f(y^{k+1})\rangle\leq0$, owing to the monotonicity  of $\nabla f$ (namely $\langle y-x, \nabla f(y)-\nabla f(x)\rangle\geq0$), we can get $\langle y^{k+1}-x^{k}, \nabla f(x^{k})\rangle\leq0$; then $y^{k+1}-x^{k}$ is a decreasing direction at point $x^{k}$ for function $f(x)+\lambda\|x\|_0$ in subspace $C_{I(x^k)}$ and hence  we  think it is worth doing extrapolation; Otherwise we reset $y^{k+1}=x^k$.  And we using $\nabla f(y^{k+1})$ rather than $\nabla f(x^{k})$ to reduce the amount of computation because $\nabla f(y^{k+1})$ is used to evaluate the next iteration point $x^{k+1}$.

\begin{remark}
In the numerical experiment, one can take an appropriate selection of parameters $\omega_k$ such that $y^{k+1}$ is always in the set $X$.
\end{remark}

\subsection{Convergence analysis}

In this section, we present the convergence results of Algorithm 1. Firstly we give some properties about the solutions of the subproblem (\ref{EPIHT}) and show that $I(x^k)$, the zero element index set of iteration sequence $x^k$, changes finitely often.
For the subproblem (\ref{EPIHT}), it has separable structure since $X$ is a box constraint. So we just need discussing the property of the following problem's solution
\begin{equation}\label{1dim}
\arg\min_{\tilde{x}\in \bR} h(\tilde{x}):=\delta_{ \tilde{X}}(\tilde{x})+\lambda\|\tilde{x}\|_0+\frac{1}{2}(\tilde{x}-c)^2
\end{equation}
where $\tilde{X}:=\{\tilde{x}\in \mathbb{R}:\tilde{l}\leq \tilde{x}\leq \tilde{u}\}$. In fact, the minimum point of function
$\delta_{\tilde{X}}(\tilde{x})+\lambda+\frac{1}{2}(\tilde{x}-c)^2$ is $P_{\tilde{X}}(c)$ and it has different function value only at zero point compared with $h(\tilde{x})$. When $P_{\tilde{X}}(c)\neq 0$, we only need to compare the function value $h(0)$ and $h(P_{\tilde{X}}(c))$ to get the solution. In detail,
\begin{itemize}
\item For the case $0\notin \tilde{X}$, the solution point is certainly  $P_{\tilde{X}}(c)$.
  \item For the case $0\in \tilde{X}$,  $h(P_{\tilde{X}}(c))-h(0)=\lambda+\frac{1}{2}(P_{\tilde{X}}(c))^2-cP_{\tilde{X}}(c)$. If $c\in \tilde{X}$, the solution point is $\mathcal H_{\sqrt{2\lambda}}(P_{\tilde{X}}(c))$ since $h(P_{\tilde{X}}(c))-h(0)=\lambda-\frac{1}{2}c^2$; If $c>\tilde{u}$, the solution is obtained by comparing  $h(0)$ and $h(\tilde{u})$; If $c<\tilde{l}$, the solution is obtained by comparing  $h(0)$ and $h(\tilde{l})$.
\end{itemize}
In either case above, the solution point $\tilde{x}^*$ satisfies $|\tilde{x}^*|\geq \min(\{|l|,|u|,\sqrt{2\lambda}\}/\{0\}^C)$ if it is not zero, where $\{0\}^C$ denotes the complement of set \{0\}. Then we have the following results.
\begin{lemma}\label{lem:nochange}
Let $H(x)$ be the objective function defined in \eqref{eqn:functionH}, and $\{x_k\}_{k=0}^\infty$ be the sequence generated by  Algorithm 1. If the extrapolation weight $\omega_k$  satisfies $0\leq\omega_k\leq \omega<1$, then
\begin{enumerate}
\item $\{H(x^k)\}_{k=0}^{+\infty}$ is non-increasing;
  \item $\sum_{k=1}^\infty\|x^k-y^{k}\|^2<\infty$, $\|x^k-y^{k}\|^2\rightarrow0$;
  \item $I(x^k)$ changes only finitely often;
  \item $\sum_{k=1}^\infty\|x^k-x^{k-1}\|^2<\infty$, $\|x^k-x^{k-1}\|^2\rightarrow0$.
\end{enumerate}
\end{lemma}
\begin{proof}
1. Since $\nabla f(x)$ is $L$-Lipschitz continuous, from Lemma 2.3, we have
\begin{equation}\label{Lip7}
f(x^{k+1})-f(y^{k+1})\leq
\langle\nabla f(y^{k+1}),x^{k+1}-y^{k+1}\rangle+\frac{L}{2}\|x^{k+1}-y^{k+1}\|^2.
\end{equation}
It is clear that $y^{k+1}\in X$. Then from Algorithm 1, we have
\begin{eqnarray*}
&&\lambda\|x^{k+1}\|_0+
\frac{L}{2}\|x^{k+1}-y^{k+1}+\frac{\nabla f(y^{k+1})}{L}\|^2+\frac{\mu}{2}\|x^{k+1}-y^{k+1}\|^2\\
&\leq& \lambda\|y^{k+1}\|_0+
\frac{L}{2}\|y^{k+1}-y^{k+1}+\frac{\nabla f(y^{k+1})}{L}\|^2.\\
\end{eqnarray*}
By summing up the above two inequalities and using the fact $\|y^{k+1}\|_0\leq\|x^k\|_0$, $f(y^{k+1})-f(x^k)\leq\langle y^{k+1}-x^{k}, \nabla f(y^{k+1})\rangle\leq0$, we have
\begin{equation}\label{e8}
H(x^{k+1})+\frac{\mu}{2}\|x^{k+1}-y^{k+1}\|^2\leq H(x^k).
\end{equation}
It is obvious that $\{H(x^k)\}_{k=0}^{+\infty}$ is non-increasing.

2. Summing the inequality \eqref{e8} over $k=0,\ldots,n-1$, we have
\[
H(x^n)+\frac{\mu}{2}\sum_{k=0}^n\|x^k-y^{k}\|^2\leq H(x^0).
\]
So $\{\sum_{k=0}^n\|x^k-y^{k}\|^2\}$ has upper bound since $H(x)$ has lower bound on $X$. Then $\sum_{k=0}^\infty\|x^k-y^{k}\|^2<\infty$ and $\|x^k-y^{k}\|\rightarrow 0$.

3. From the implementation of iteration $k$ in  Algorithm 1 and the discussion about the property of problem (\ref{1dim})'s solution , we have
\[
|x^k_i|\geq ls:=\min(\{|l_j|,|u_j|,\sqrt{\frac{2\lambda}{L+\mu}},j=1,\cdots,n\}\cap\{0\}^C), \mbox{ for any } i\not\in I(x^k).
 \]
 Hence, we have $\|x^k-y^{k}\|\geq ls>0$ if $I(x^k)^C\nsubseteq I(y^{k})^C$. From $\|x^k-y^{k}\|\rightarrow 0$ and $I(y^k)^C\subseteq I(x^{k-1})^C$, it is easy to see that $I(x^k)^C\subseteq I(x^{k-1})^C$ always hold if $k$ is sufficiently large. Then
$I(x^k)$ must change only finitely often.

4. Assume that $I(x^k)=I(x^{k+1})$ for any $k\geq k_0$. From the subproblem (\ref{EPIHT}), we have
\begin{eqnarray*}
&&\lambda\|x^{k+1}\|_0+
\frac{L}{2}\|x^{k+1}-y^{k+1}+\frac{\nabla f(y^{k+1})}{L}\|^2+\frac{\mu}{2}\|x^{k+1}-y^{k+1}\|^2\\
&\leq& \lambda\|x\|_0+
\frac{L}{2}\|x-y^{k+1}+\frac{\nabla f(y^{k+1})}{L}\|^2+\frac{\mu}{2}\|x-y^{k+1}\|^2\\
\end{eqnarray*}
for any $x\in C:=C_{I(x^k)}\cap X$ and $k\geq k_0$. So we have
\begin{equation}\label{subPro}
x^{k+1}\in \arg\min_{x \in C} Q(x;y^{k+1}):=\frac{L}{2}\|x-y^{k+1}+\frac{1}{L}\nabla f(y^{k+1})\|^2+\frac{\mu}{2}\|x-y^{k+1}\|^2,\;\; k\geq k_0.
\end{equation}
From the optimality condition we have $0\in \partial \delta_C(x^{k+1})+\nabla Q(x^{k+1};y^{k+1})$, namely
$-\nabla Q(x^{k+1};y^{k+1})\in \partial \delta_C(x^{k+1})$. Hence for any $x\in C$, the following inequality holds
\[
0\geq \langle -\nabla Q(x^{k+1};y^{k+1}),x-x^{k+1}\rangle.
\]
Using the above inequality and the strong convexity of $Q(x;y^{k+1})$ with modulus $L+\mu$, we have
\begin{eqnarray*}
&Q(x^{k};y^{k+1})&\geq Q(x^{k+1};y^{k+1})+\langle \nabla Q(x^{k+1};y^{k+1}),x^k-x^{k+1}\rangle+\frac{L+\mu}{2}\|x^{k+1}-x^k\|^2\\
&&\geq Q(x^{k+1};y^{k+1})+\frac{L+\mu}{2}\|x^{k+1}-x^k\|^2.
\end{eqnarray*}
Combining the above inequality and (\ref{Lip7}), we obtain
\begin{eqnarray}\label{e4}
&&f(x^{k+1})+\frac{\mu}{2}\|x^{k+1}-y^{k+1}\|^2+\frac{L+\mu}{2}\|x^{k+1}-x^{k}\|^2\\
&\leq& f(y^{k+1})+\langle \nabla f(y^{k+1}), x^k-y^{k+1}\rangle+ \frac{L+\mu}{2}\|y^{k+1}-x^{k}\|^2\\
&\leq& f(x^k)+\frac{(L+\mu)\omega_k^2}{2}\|x^{k-1}-x^{k}\|^2.
\end{eqnarray}
Summing the above inequality over $k=1,\ldots,n,\ldots$, we have
\begin{eqnarray*}
&\sum_{k=0}^\infty\frac{(L+\mu)(1-\omega^2)}{2}\|x^k-x^{k-1}\|^2
&\leq\sum_{k=0}^\infty\frac{(L+\mu)(1-\omega_k^2)}{2}\|x^k-x^{k-1}\|^2\\
&& \leq f(x^0)-\min_{x\in X} f(x)<\infty.
\end{eqnarray*}
Then $\sum_{k=1}^\infty\|x^k-x^{k-1}\|^2<\infty$ and $\|x^k-x^{k-1}\|^2\rightarrow0$.
\end{proof}

In the following, we establish the convergence of $\{x^k\}_{k=0}^\infty$.
\begin{theorem}\label{thm:kexists}
Let $H(x)$ be the objective function defined in \eqref{eqn:functionH}, and $\{x^k\}_{k=0}^\infty$ be the sequence generated by Algorithm 1, then
\begin{enumerate}
  \item $x^k$ is bounded;
  \item any cluster point of $\{x^k\}_{k=0}^\infty$ is a local minimizer of $H(x)$;
      \item $H(x^k)\rightarrow H(x^*)$ where $x^*$ is a cluster point of $\{x^k\}_{k=0}^\infty$;
  \item if $\omega_k\equiv \omega\in(0,1)$, $x^k$ is convergent.

\end{enumerate}
\end{theorem}
\begin{proof}

1.  It is clear that $x^k\in X$ is bounded.

2. Assume that $x^*$ is a cluster point of $\{x^k\}_{k=0}^\infty$ and the subsequence $
 x^{k_j}$ converging to $x^*$. From Lemma~\ref{lem:nochange}, $\|x^k-x^{k-1}\|\rightarrow 0$ and $I(x^k)$ changes only finitely often. So we have $\|x^k-y^{k+1}\|\rightarrow 0$ and there exists $k_0$ such that for any $k\geq k_0$,  $I(y^{k+1})=I(x^k)=I(x^{k+1})=I(x^*)$.

From  (\ref{subPro}), we have
  \[
x^{k+1}=P_{C_{I(x^*)}\cap X}(y^{k+1}-\frac{\nabla f(y^{k+1})}{L+\mu}), \;\; k\geq k_0.
\]
 Letting  $k$ be equal to $k_j$ and $j$ tend to infinity, from the continuity of projection operator, we obtain
\begin{eqnarray*}
&x^*=P_{C_{I(x^*)}\cap X}(x^*-\frac{\nabla f(x^*)}{L+\mu})
\end{eqnarray*}
Since $X$ is a box constraint, we have
\[
x^*_i=P_{C_{I(x^*_i)}\cap [l_i,u_i]}(x^*_i-\frac{(\nabla f(x^*))_i}{L+\mu}).
\]
From the proof of Lemma \ref{lem:nochange}, if $x^*_i\neq 0$, $C_{I(x^*_i)}=\bR$ and $x^*_i=P_{[l_i,u_i]}(x^*_i-\frac{(\nabla f(x^*))_i}{L+\mu})$.
From the property of projection operator, we have $(x_i-x^*_i)(x^*_i-\frac{(\nabla f(x^*))_i}{L+\mu}-x^*_i)\leq0 $, namely $(x_i-x^*_i)(\nabla f(x^*))_i\geq0$, for any $x\in X$.

Denote
\[
U:=\{\Delta x:\Delta x+x^*\in X, \|\Delta x\|_{\infty}<\min_{i\in I(x^*)}\{\frac{\lambda}{|(\nabla f(x^*))_i|}\}; |\Delta x_i|<|x^*_i|,\forall i\notin I(x^*) \},
 \]
Then for any
$\Delta x\in U$, we have
\begin{itemize}
\item $\sum_{i\not\in I(x^*)}\lambda\|x^*_i+
\Delta x_i\|_0=\sum_{i\not\in I(x^*)}\lambda\|x^*_i\|_0,\;\forall i\notin I(x^*)$ since $|\Delta x_i|<|x^*_i|$;
           \item $\Delta x_i(\nabla f(x^*))_i\geq0,\;\forall i\notin I(x^*)$ since $x^*+\Delta x\in X$;
           \item $\lambda\|\Delta x_i\|_0+(\nabla f(x^*))_i\Delta x_i\geq0,\; i\in I(x^*)$ since
               \[
               \|\Delta x\|_{\infty}<\min_{i\in I(x^*)}\{\lambda/|(\nabla f(x^*))_i|\}.
               \]
                Furthermore if $\Delta x_i\neq 0$, it is clear that $\lambda\|\Delta x_i\|_0+(\nabla f(x^*))_i\Delta x_i>0$.
         \end{itemize}

         From the above conclusions, for any $\Delta x \in U $,  we have
\begin{eqnarray*}
&H(x^*+\Delta x)-H(x^*)&=\lambda\|x^*+\Delta x\|_0-\lambda\|x^*\|_0+f(x^*+\Delta x)-f(x^*) \\
&& \geq \sum_{i\in I(x^*)}\lambda\|
\Delta x_i\|_0+\langle \nabla f(x^*), \Delta x\rangle\\
&&=\sum_{i\in I(x^*)}(\lambda\|\Delta x_i\|_0+(\nabla f(x^*))_i\Delta x_i)+\sum_{i\not\in I(x^*)}\Delta x_i(\nabla f(x^*))_i\\
&&\geq0
\end{eqnarray*}
It is clear that $x^*+U$ is a neighborhood of $x^*$. So $x^*$ is a local minimizer of objective function $H(x)$.

3. From the inequality (\ref{e8}),
$H(x^{k+1})$ is non-increasing. $f(x)$ is bounded from below on $X$,  so we have
$H(x^k)$ is convergent.
Furthermore, $H(x^{k})\rightarrow H(x^*)$ since $I(x^k)=I(x^*)$ when $k\geq k_0$ and $H(x^{k_j})\rightarrow H(x^*)$.

4.  As we have known,  for any $k\geq k_0$ and $x\in C:=C_{I(x^*)}\cap X$, we have
\begin{eqnarray*}
&x^{k+1}&=P_C(y^{k+1}-\frac{\nabla f(y^{k+1})}{L+\mu}),
\end{eqnarray*}
and $I(y^{k+1})=I(x^k)=I(x^{k+1})=I(x^*)$.
Then using Lemma \ref{Lipschitz1}, for any $x\in C$, we can obtain
\begin{equation}\label{Subpro2}
f(x^{k+1})\leq f(x)+(L+\mu)\langle x-y^{k+1}, x^{k+1}-y^{k+1}\rangle-\frac{L+\mu}{2}\|x^{k+1}-y^{k+1}\|^2,k\geq k_0.
\end{equation}
Setting $x=x^*$ and $x=x^{k}$ respectively, we have
\begin{equation}\label{e3}
f(x^{k+1})\leq f(x^*)+(L+\mu)\langle x^*-y^{k+1}, x^{k+1}-y^{k+1}\rangle-\frac{L+\mu}{2}\|x^{k+1}-y^{k+1}\|^2,k\geq k_0,
\end{equation}
\begin{equation}\label{e2}
f(x^{k+1})\leq f(x^{k})+(L+\mu)\langle x^k-y^{k+1}, x^{k+1}-y^{k+1}\rangle-\frac{L+\mu}{2}\|x^{k+1}-y^{k+1}\|^2,k\geq k_0.
\end{equation}
Note that either $y^{k+1}=x^{k}+\omega(x^{k}-x^{k-1})$ or $y^{k+1}=x^k$, the above inequality always holds.

\textbf{Case 1:} there exists $k_1>k_0$ such that $y^{k+1}=x^{k}+\omega(x^{k}-x^{k-1})$ for any $k\geq k_1$.

Multiplying the inequality (\ref{e3}) by $1-\omega$ and (\ref{e2}) by
$\omega$, then adding the two resulting inequalities, and using the fact $f(x^*)\leq f(x^k)$, we obtain, for any $k\geq k_1$,
\begin{eqnarray*}
&&\frac{2}{ L+\mu}(f(x^{k+1})-f(x^k))\\
&\leq&2\langle (1-\omega)x^*+\omega x^k-y^{k+1}, x^{k+1}-y^{k+1}\rangle-\|x^{k+1}-y^{k+1}\|^2\\
&=&\|y^{k+1}-\omega x^{k}-(1-\omega)x^*\|^2-\|x^{k+1}-\omega x^k-(1-\omega)x^*\|^2\\
&=&\|x^{k}-\omega x^{k-1}-(1-\omega)x^*\|^2-\|x^{k+1}-\omega x^k-(1-\omega)x^*\|^2.\\
\end{eqnarray*}
This implies $\{\frac{2}{ L+\mu}f(x^{k})+\|x^{k}-\omega x^{k-1}-(1-\omega)x^*\|^2\}_{k\geq k_1}$ is a non-increasing sequence. So it is convergent.
Noting that $x^*$ is a cluster point of $\{x^k\}_{k=0}^{+\infty}$, $f(x^k)\rightarrow f(x^*)$ and $x^{k}-x^{k-1}\rightarrow0$, we can obtain $\|x^{k}-x^*\|^2\rightarrow0$.

\textbf{Case 2:} for any $k_1>k_0$, there exists $k>k_1$  such that $y^{k+1}=x^{k}$.

For simplicity, denote $\sigma_n:=\frac{2}{L+\mu}f(x^{n})$ and $\rho_n:=\|x^{n}-x^{n-1}\|^2$.
If  $y^{n+1}=x^{n}$, from the inequality (\ref{e3}), we obtain
\begin{equation}\label{e6}
0\leq\frac{2}{ L+\mu}(f(x^{n+1})-f(x^*))\leq\|x^n-x^*\|^2-\|x^{n+1}-x^*\|^2.
\end{equation}
Combing it with inequality \eqref{e4}, we have
\begin{equation}\label{e7}
\sigma_{n+1}+\rho_{n+1}+(1-\omega)^2\|x^{n+1}-x^*\|^2\leq \sigma_n+\omega^2\rho_n
+(1-\omega)^2\|x^{n}-x^*\|^2.
\end{equation}
If $y^{n+1}=x^{n}+\omega(x^{n}-x^{n-1})$, from the discussion in case 1, we have
\begin{equation}\label{e9}
\sigma_{n+1}+\|x^{n+1}-\omega x^n-(1-\omega)x^*\|^2\leq \sigma_{n}+
\|x^{n}-\omega x^{n-1}-(1-\omega)x^*\|^2.
\end{equation}

Without loss of generality, we assume that
\begin{eqnarray*}
&&y^{k_0+1}=x^{k_0}, \cdots, y^{k_1}=x^{k_1-1},\\
&&y^{k_1+1}=x^{k_1}+\omega(x^{k_1}-x^{k_1-1}), \cdots, y^{k_2}=x^{k_2-1}+\omega(x^{k_2-1}-x^{k_2-2}),\\
&&y^{k_2+1}=x^{k_2}, \cdots, y^{k_3}=x^{k_3-1},\\
\end{eqnarray*}
and this happens again and again. So we just need discuss for $k_1\leq k\leq k_3-1$.
Form the inequality \eqref{e6}, \eqref{e7} and \eqref{e9}, we can obtain
\[
\sigma_{k_0+1}+\rho_{k_0+1}+(1-\omega)^2\|x^{k_0+1}-x^*\|^2\leq \sigma_{k_0}+\omega^2\rho_{k_0}+
(1-\omega)^2\|x^{k_0}-x^*\|^2
\]
$$\vdots$$
\[
\sigma_{k_1}+\rho_{k_1}+(1-\omega)^2\|x^{k_1}-x^*\|^2\leq \sigma_{k_1-1}+\omega^2\rho_{k_1-1}+
(1-\omega)^2\|x^{k_1-1}-x^*\|^2
\]
\[
\sigma_{k_1+1}+\|x^{k_1+1}-\omega x^{k_1}-(1-\omega)x^*\|^2\leq\sigma_{k_1}+
\|x^{k_1}-\omega x^{k_1-1}-(1-\omega)x^*\|^2
\]
$$\vdots$$
\[
\sigma_{k_2}+\|x^{k_2}-\omega x^{k_2-1}-(1-\omega)x^*\|^2\leq \sigma_{k_2-1}+
\|x^{k_2-1}-\omega x^{k_2-2}-(1-\omega)x^*\|^2
\]
\[
\sigma_{k_2+1}+\rho_{k_2+1}+(1-\omega)^2\|x^{k_2+1}-x^*\|^2\leq \sigma_{k_2}+\omega^2\rho_{k_2}+
(1-\omega)^2\|x^{k_2}-x^*\|^2
\]
$$\vdots$$
\[
\sigma_{k_3}+\rho_{k_3}+(1-\omega)^2\|x^{k_3}-x^*\|^2\leq \sigma_{k_3-1}+\omega^2\rho_{k_3-1}+
(1-\omega)^2\|x^{k_3-1}-x^*\|^2
\]
We denote the terms on the right side of the above inequalities as sequence $\{u^k\}_{k\geq k_0}^{k_3-1}$.
It's clear that $\{u^k\}_{k\geq k_0}^{k_1-1}$, $\{u^k\}_{k\geq k_1}^{k_2-1}$ and $\{u^k\}_{k\geq k_2}^{k_3-1}$ is non-increasing.
For the following situation
\[
\sigma_{k_2}+\|x^{k_2}-\omega x^{k_2-1}-(1-\omega)x^*\|^2\leq \sigma_{k_2-1}+
\|x^{k_2-1}-\omega x^{k_2-2}-(1-\omega)x^*\|^2
\]
\[
\sigma_{k_2+1}+\rho_{k_2+1}+(1-\omega)^2\|x^{k_2+1}-x^*\|^2\leq \sigma_{k_2}+\omega^2\rho_{k_2}+
(1-\omega)^2\|x^{k_2}-x^*\|^2,
\]
if $\langle x^{k_2}-x^{k_2-1},  x^{k_2}-x^*\rangle\geq0 $, then
\begin{eqnarray*}
&u^{k_2-1}&\geq \sigma_{k_2}+\|x^{k_2}-\omega x^{k_2-1}-(1-\omega)x^*\|^2\\
&&\geq\sigma_{k_2}+\omega^2\|x^{k_2}-x^{k_2-1}\|^2+(1-\omega)^2\|x^{k_2}-x^*\|^2\\
&&= u^{k_2}
\end{eqnarray*}
and the sequence $\{u^k\}_{k\geq k_1}^{k_3-1}$ is non-increasing;
otherwise
\begin{eqnarray*}
&\|x^{k_2-1}-x^*\|^2&=\|x^{k_2-1}-x^{k_2}\|^2+\|x^{k_2}-x^*\|^2-2\langle x^{k_2}-x^{k_2-1},  x^{k_2}-x^*\rangle\\
&&\geq \|x^{k_2}-x^*\|^2,
\end{eqnarray*}
combing it with inequality \eqref{e4}, we have
\[
\sigma_{k_2}+\rho_{k_2}+(1-\omega)^2\|x^{k_2}-x^*\|^2\leq \sigma_{k_2-1}+\omega^2\rho_{k_2-1}+
(1-\omega)^2\|x^{k_2-1}-x^*\|^2,
\]
then we redefine $u^{k_2-1}:=\sigma_{k_2-1}+\omega^2\rho_{k_2-1}+
(1-\omega)^2\|x^{k_2-1}-x^*\|^2$ and hence $u^{k_2-1}\geq u^{k_2}$,
repeating the above process for $\{u^k\}_{k\geq k_2-1}^{k_1}$ and redefine $u^k$ if necessary, we can obtain a non-increasing sequence
$\{u^k\}_{k\geq k_1}^{k_2-1}$. If $u^{k_1}$ isn't redefined, the following situation happens
\[
\sigma_{k_1}+\rho_{k_1}+(1-\omega)^2\|x^{k_1}-x^*\|^2\leq \sigma_{k_1-1}+\omega^2\rho_{k_1-1}+
(1-\omega)^2\|x^{k_1-1}-x^*\|^2:=u^{k_1-1}
\]
\[
\sigma_{k_1+1}+\|x^{k_1+1}-\omega x^{k_1}-(1-\omega)x^*\|^2\leq\sigma_{k_1}+
\|x^{k_1}-\omega x^{k_1-1}-(1-\omega)x^*\|^2:=u^{k_1}.
\]
Noting that $\|x^{k_1}-x^*\|^2\leq \|x^{k_1-1}-x^*\|^2$, hence
\begin{eqnarray*}
&u^{k_1}&= \sigma_{k_1}+\omega\|x^{k_1}-x^{k_1-1}\|^2+
(1-\omega)\|x^{k_1}-x^*\|^2-\omega(1-\omega)\|x^{k_1-1}-x^*\|^2\\
&&\leq \sigma_{k_1}+\|x^{k_1}-x^{k_1-1}\|^2+
(1-\omega)^2\|x^{k_1}-x^*\|^2\\
&&\leq \sigma_{k_1-1}+\omega^2\rho_{k_1-1}+
(1-\omega)^2\|x^{k_1-1}-x^*\|^2\\
&&=u^{k_1-1};
\end{eqnarray*}
otherwise, the following situation happens
\[
\sigma_{k_1}+\rho_{k_1}+(1-\omega)^2\|x^{k_1}-x^*\|^2\leq \sigma_{k_1-1}+\omega^2\rho_{k_1-1}+
(1-\omega)^2\|x^{k_1-1}-x^*\|^2:=u^{k_1-1}
\]
\[
\sigma_{k_1+1}+\rho_{k_1+1}+(1-\omega)^2\|x^{k_1+1}-x^*\|^2\leq \sigma_{k_1}+\omega^2\rho_{k_1}+
(1-\omega)^2\|x^{k_1}-x^*\|^2:=u^{k_1}
\]
and it's clear that $u^{k_1-1}\geq u^{k_1}$.
In summary, we can obtain a non-increasing sequence $\{u^k\}_{k\geq k_0}^{k_3-1}$ where $u^k=\sigma_{k}+\omega^2\rho_{k}+
(1-\omega)^2\|x^{k}-x^*\|^2$ or $u^k=\sigma_{k}+\|x^{k}-\omega x^{k-1}-(1-\omega)x^*\|^2$. Repeating this process, we finally obtain
a non-increasing sequence $\{u^k\}_{k\geq k_0}^{+\infty}$. So it's convergent. Combining the fact
 $x^*$ is a cluster point of $\{x^k\}_{k=0}^{+\infty}$, $\sigma_k\rightarrow \frac{2}{L+\mu}f(x^{*})$ and $x^{k}-x^{k-1}\rightarrow0$, we can obtain that
 $\|x^k-x^*\|\rightarrow0$.
\end{proof}

\section{Discussions}
\label{sec:dis}
Recently, some extrapolation type methods were proposed for $\ell_0$ regularization problem or more general non-convex problems. In particular,
 the inertial forward-backward (IFB) method \cite{bot2014inertial} for
 solving problem (\ref{fplusg})(both $f(x)$ and $g(x)$ can be non-convex) uses  Bregman distance. Under
  Kurdyka-{\L}ojasiewicz property  theoretical framework, the sequence generated by IFB method converges to a critical point when $f+g$ is coercive.
  When we take the Bregman distance function  as $\|\cdot\|^2/2$, IFB method is the algorithm proposed by \cite{Ochs2014iPiano} while $g(x)$ needs to be convex. If we
 apply IFB method to $\ell_0$ regularization problem (\ref{fl0}), the iterative scheme is
\begin{eqnarray*}
&&y^{k+1}=x^{k}+2\beta_k(x^k-x^{k-1})\\
&&x^{k+1}\in \arg\min_{x \in X}\lambda\| x\|_0+\frac{1}{4\alpha_k}\|x-x^{k}+2\alpha_k\nabla f(x^{k})\|^2+\frac{1}{4\alpha_k}\|x-y^{k+1}\|^2
\end{eqnarray*}
where $\alpha_k,\beta_k>0$ satisfy
\begin{eqnarray}
&&0<\underline{\alpha}\leq\alpha_k\leq \overline{\alpha},0<\beta_k\leq\beta\\ \label{IFB1}
&&1>\overline{\alpha}L+2\beta\frac{\overline{\alpha}}{\underline{\alpha}}. \label{IFB2}
\end{eqnarray}
for some  $\overline{\alpha},\underline{\alpha},\beta>0$ and $\nabla f$'s Lipschitz constant $L$.
It is easy to see that when $\beta_k\equiv0$, IFB becomes PIHT. Usually a larger $\alpha_k$ leads to a  faster convergence. However, the above inequality \eqref{IFB2} implies that  a larger $\alpha_k$ leads to a small $\beta_k$, thus the extrapolation step will have small effect on the speed of IFB method. In other words, one  cannot have both of large $\alpha_k$ and $\beta_k$. This limits the acceleration effect of IFB method  against PIHT method.

The extrapolated PIHT (EPIHT) method \cite{BDHSZZ2016} is  proposed for solving
\[
\min_{x\in\bR^n} H(x):= \lambda\| x\|_0+f(x)+\frac{t}{2}\|x\|^2,
\]
where $f$ is convex and $\nabla f$ is Lipschitz continuous, could have both large step size and large extrapolated step size. Its iterative scheme takes the form
\begin{eqnarray*}
&&y^{k+1}=x^{k}+\omega_k(x^{k}-x^{k-1})\\
&&\mbox{if } H(y^{k+1})>H(x^{k}), \;\; \mbox{reset } y^{k+1}=x^k\\
&&x^{k+1}\in \arg\min_{x \in\bR^n}\lambda\| x\|_0+\frac{L+t}{2}\|x-y^{k+1}+\frac{\nabla f(y^{k+1})}{L+t}\|^2+\frac{\mu}{2}\|x-y^{k+1}\|^2
\end{eqnarray*}
where $\mu>0$, $0<\omega_k\leq\omega<1$. It is similar with the IFB method except that the linearization is performed at $y^{k+1}$ instead of $x^k$ and the setting for parameters is also different.
Under Kurdyka-{\L}ojasiewicz property theoretical framework, the sequence generated by EPIHT method globally converges to a local minimizer of $H(x)$.

For more general problem (\ref{fplusg}) (both $f(x)$ and $g(x)$  can be non-convex), \cite{chubulai} proposed the following monotone accelerated proximal gradient (mAPG) method
\begin{eqnarray*}
&&y^k=x^k+\frac{t_{k-1}}{t_k}(z^k-x^k)+\frac{t_{k-1}-1}{t_k}(x^k-x^{k-1})\\
&&z^{k+1}=\arg\min_zg(z)+\frac{1}{2\alpha_y}\|z-y^k+\alpha_y\nabla f(y^k)\|^2\\
(\mbox{mAPG})\quad&&v^{k+1}=\arg\min_vg(v)+\frac{1}{2\alpha_x}\|v-x^k+\alpha_x\nabla f(x^k)\|^2\\
&&t_{k+1}=\frac{\sqrt{1+4t_k^2}+1}{2}\\
&&x^{k+1}=
\left\{
\begin{array}{cl}
  z^{k+1},& \mbox{ if }f(z^{k+1})+g(z^{k+1})\leq f(v^{k+1})+g(v^{k+1}) \\
  v^{k+1},& \mbox{ otherwise}
\end{array}
\right.
\end{eqnarray*}
When $f(x), g(x)$ are convex, mAPG has $O(1/k^2)$~convergence rate; otherwise, any cluster point of iteration sequence is a critical point of $f(x)+g(x)$. Based on mAPG, \cite{chubulai} also proposed a non-monotone APG(nmAPG) for saving the computation cost in each step.
Denote
\[
q_1=1,c_1=f(x_1)+g(x_1);q_{k+1}=\eta q_k+1, c_{k+1}=\frac{\eta q_kc_k+f(x^{k+1})+g(x^{k+1})}{q_{k+1}}.
\]
If $f(z^{k+1})\leq c_k-\delta\|x^{k+1}-y^k\|^2$, nmAPG gets the next iteration point by $x^{k+1}=z^{k+1}$, otherwise, it gets the next iteration point same with mAPG.

Moreover, \cite{Pock2017Inertial} proposed an inertial proximal alternating linearized minimization (iPALM) method for solving problem
\[
\min_{x,y}s(x)+q(x,y)+r(y).
\]
The iterative sequence has global convergence.  If $r(y)\equiv 0$, $s(x)=\lambda\|x\|_0$ and $q(x,y)=f(x)$, the above problem reduces to
\eqref{fl0}, and iPALM is simplified as
\begin{eqnarray*}
&& y^k=x^k+\alpha_k(x^k-x^{k-1})\\
&& z^{k}=x^k+\beta_k(x^k-x^{k-1})\\
&& x^{k+1}=\arg\min_{x\in X} \lambda\|x\|_0+\frac{1}{2\tau_k}\|x-y^k+\tau_k\nabla f(z^k)\|^2
\end{eqnarray*}
If the objective function satisfies Kurdyka-{\L}ojasiewicz property, the iteration sequence has global convergence, but the extrapolated step length $\alpha_k, \beta_k$, and the proximal parameter $\tau_k$ need to satisfy an equation.

\begin{table}
\centering
\begin{tabular}{|c|c|c|c|c|c|}
\hline
Method& Assumption &  Parameters& Convergence &NCf&NCGf\\
\hline
IFB& nonconvex $f$, KL & BC+inequalities & globally&0&1 \\
EPIHT& convex $f$, KL & BC &globally&2&1 \\
mAPG& nonconvex $f$ & BC &subsequence&2&2 \\
iPALM&nonconvex $f$, KL & BC+equation  &globally&0&1\\
Our& convex $f$ & BC &globally&0&1 or 2 \\\hline
\end{tabular}
\caption{Comparisons of IFB, EPIHT, mAPG, iPALM and our method. Parameters refer to $\alpha_k, \beta_k$ in IFB, $\omega_k, \mu$ in EPIHT, $\alpha_x, \alpha_y$ in mAPG, $\alpha_k, \beta_k, \tau_k$ in iPALM, $\omega_k, \mu$ in our method and BC represents box constraint. NCf represents the number of computation of $f$ and NCGf represents the number of computation of $\nabla f$  during one iteration.  }\label{e1}
\end{table}
In Table \ref{e1}, we summarize  some  differences of the above mentioned algorithms and our method. All the methods require $f$ being differentiable and $\nabla f$ being  Lipschitz continuous and this is not stated again in the table. We point out that:
\begin{itemize}
  \item  our method's global convergence analysis does not rely on  KL property;
  \item the constraint conditions of parameters of EPIHT, mAPG and our method are relatively simpler compared to other methods;
  \item IFB and iPALM methods need the least amount of computation per iteration,  but the conditions on the algorithm parameters are more complex or  restricted,  which could increase the total number of iteartions;
  \item compared with EPIHT and mAPG, our method need less computation cost for one iteration; This can cost less computation time when the  total iteration number is fixed.
\end{itemize}
\section{Numerical Implementation}
\label{sec:test}
In this section, we will show some numerical results of Algorithm 1 on $\ell_0$ minimization problems \eqref{eqn:functionH}, and compare with the results of PIHT, IFB, mAPG, nmAPG, EPIHT methods. All the experiments are conducted in MATLAB using a desktop computer equipped with a $4.0$GHz $8$-core AMD processor and $16$GB memory.

\subsection{Compressive sensing}
\label{sec:cs}
We first test the algorithms on  a standard sparse signal reconstruction problem in compressive sensing \cite{EJ06}. The goal is to reconstruct a sparse signal from a set of noisy linear measurements. The  following $\ell_0$ regularization formulation can be considered
\begin{equation}\label{exsig2}
\min_{x\in X }\frac{1}{2}\|Ax-b\|^2+\lambda\|x\|_0
\end{equation}
where $A\in \mathbb{R}^{m\times n}$ is a data matrix, $b\in \mathbb{R}^m$ is an observation vector, and $X=\{x\in \bR^n|-10^{10}\leq x\leq 10^{10}\}$.
We set $f(x)=\|Ax-b\|^2/2$ and then the Lipschitz constant of $\nabla f(x)$ is $L=\lambda_{\max} (A^\top A)$, where $\lambda_{\max} (A^\top A)$ denotes the maximum eigenvalue of $A^\top A$.

For this experiment, the data matrix $A\in \bR^{m\times n}$ is a Gaussian random matrix and the columns of $A$ are normalized to have $\ell_2$ norm of $1$. We set $m=3000$ and test on different size of $n$ and sparsity level $s$ of the unknown signal.  For each choice of $(n,s)$, we generate the true signal $\bar{x}\in\bR^n$ containing $s$ randomly
placed $\pm 1$ spikes. The observed data $b\in \bR^m$ is generated by
\[
b=Ax+\eta,
\]
where $\eta$ is a white Gaussian noise of variance $0.05$. And for each pair of $(n,s)$, we run our experiment $50$ times to guarantee that the result is independent of any particular realization of the random matrix and true signal $\bar{x}$.

For all the methods, the stopping criteria is commonly set to be
\[
\frac{\|x_k-x_{k-1}\|}{\mbox{max}\{1,\|x_k\|\}}<10^{-5},
\]
and the initial point is obtained by FISTA\cite{BT09} for $\ell_1$ minimization (where the initial point is $x_0=A^Tb$, and the stopping criteria is $\frac{\|x_k-x_{k-1}\|}{\mbox{max}\{1,\|x_k\|\}}<10^{-2}$, the regularization  parameter $\lambda=0.1$) and the corresponding iteration number and running time are added in the final results.
All the parameters are chosen according to empirically the lowest relative error $\frac{\|x-\bar{x}\|}{\|\bar{x}\|}$. In detail, we choose $\ell_0$ regularization $\lambda_1=0.3$; choose $\mu=10^{-6}$, $\omega_k=0.99$ for Algorithm 1 and EPIHT method; choose $\beta_k=10^{-6}$, $\alpha_k=(0.999999-2\beta_k)/L$ for IFB method; choose $\alpha_x=\alpha_y=1/(L+10^{-6})$ for mAPG and nmAPG method, moreover choose $\eta=0.8$ for  nmAPG method. For each algorithm and each choice of $(n,s)$ of the solution $\bar{x}$, we conduct 50 experiments and record the average runtime, the average relative error $\frac{\|x-\bar{x}\|}{\|\bar{x}\|}$ to the original signal $\bar{x}$,
  the average number of iteration the algorithm needed and their standard variance. In fact, we find that the approximate solutions' $\ell_0$-norm of all the methods are always equal to the true $s$ and we do not list them in the tables.

The average relative errors of all the methods are very close. In fact, if the results are rounded up to $4$ decimal digits,  the results are the same as present in Table \ref{CS2}£¬especially all the results if PIHT and IFB methods.
 The average number of iterations and average runtime
are listed in Table \ref{CS1}, \ref{CS3} respectively to compare the convergence speed of different methods. We observe that:
\begin{itemize}
  \item in term of average number of iterations (Table \ref{CS1}), mAPG method , especially EPIHT and our method,
      have obvious accelerating effect  compared to PIHT;
  \item recall that the amounts of computation for each step of the algorithms are different; although EPIHT and our method have fewer, similar iteration number, our method obviously has less runtime (Table \ref{CS3});
 \item  the stability of all the methods is comparable, by  the standard variance result present in Table \ref{CS1}, \ref{CS3}.
\end{itemize}

In Table \ref{CS3}, our method has  less runtime compared to EPIHT as it requires less computation of gradient function (NCGf) per iteration.  The  average  total number of computation of gradient function is recorded  in Table \ref{e5} based on $20$ times of experiments. It can be observed that if the step \eqref{eq:restart} occurs in every iteration, the total NCGf should be two times the number of iterations. In fact, Table \ref{e5} demonstrates that the restart step \eqref{eq:restart} occurs in a low rate. Thus the extrapolation contributes to the reduction of computation and the number of  iterations.

\begin{table}
\centering
\begin{tabular}{|c|c|c|}
\hline
\multicolumn{1}{|c|}{$s$}&\multicolumn{1}{|c|}{$n$}
&\multicolumn{1}{|c|}{ Average
relative error/Standard variance}\\
\hline
&8000&0.0491/0.0040\\

 $\lfloor\frac{n}{100}\rfloor$&14000&0.0502/0.0032\\

&20000&0.0504/0.0026 \\  \cline{1-3}

&8000&0.0512/0.0030
\\
 $\lfloor\frac{2n}{100}\rfloor$&14000&0.0513/0.0024\\

 &20000&0.0521/0.0018 \\
\cline{1-3}
\end{tabular}
\caption{Results of the average and the standard variance of the
relative error.}\label{CS2}
\end{table}

\begin{table}[htb]\footnotesize
\centering
\begin{tabular}{|c|c|llllll|}
\hline
\multicolumn{1}{|c|}{$s$}&\multicolumn{1}{|c|}{$n$}
&\multicolumn{6}{|c|}{ Average
number of iteration/Standard variance}\\
\hline
&&PIHT&IFB&mAPG&nmAPG&EPIHT&Algorithm 1\\ \cline{2-8}
&8000&  55.0/1.0 & 55.0/1.0 & 42.4/0.8& 57.7/0.9& \textbf{33.7}/\textbf{0.5} & 33.9/0.8\\   \cline{2-8}
$\lfloor\frac{n}{100}\rfloor$ &14000&79.8/1.5& 79.8/1.5  & 59.8/1.0& 82.3/2.1& 44.3/\textbf{0.5} & \textbf{43.2}/1.8\\  \cline{2-8}

&20000& 105.1/1.4 & 105.1/1.4& 76.1/0.9& 103.9/1.2& 54.0/0.6& \textbf{52.8}/\textbf{0.4}\\  \cline{1-8}

&8000& 59.7/0.9& 59.7/0.9& 45.4/0.7& 61.7/0.7& \textbf{35.0}/\textbf{0.3}& 36.5/1.7\\ \cline{2-8}

$\lfloor\frac{2n}{100}\rfloor$  &14000&92.1/1.5 & 92.1/1.5 & 67.9/1.1& 93.5/1.6& 48.5/\textbf{0.5} & \textbf{47.6}/0.8\\ \cline{2-8}

 &20000&128.8/2.1& 128.8/2.1 & 91.8/1.6& 126.5/3.8& \textbf{66.5}/2.2& \textbf{66.5}/\textbf{0.8}\\  \cline{1-8}
\end{tabular}
\caption{Results of the average and the standard variance of iterations number. }\label{CS1}
\end{table}

\begin{table}[htb]\footnotesize
\centering
\begin{tabular}{|c|c|llllll|}
\hline
\multicolumn{1}{|c|}{$s$}&\multicolumn{1}{|c|}{$n$}
&\multicolumn{6}{|c|}{ Average
runtime/Standard variance}\\
\hline
&&PIHT&IFB&mAPG&nmAPG&EPIHT&Algorithm 1\\ \cline{2-8}
&8000& 4.1/0.20 & 4.1/0.20  & 6.9/0.37 &
5.9/0.28 & 3.8/0.18 & \textbf{2.9}/\textbf{0.15}\\ \cline{2-8}

$\lfloor\frac{n}{100}\rfloor$&14000& 10.1/0.21&
10.1/\textbf{0.20} & 16.3/0.32 & 14.0/0.39 & 8.1/0.07& \textbf{6.0}/0.29\\\cline{2-8}

 &20000&19.1/0.294 & 19.1/0.28  & 29.6/0.48 & 25.5/0.30 & 13.8/0.18 & \textbf{10.3}/\textbf{0.09}\\\cline{1-8}

&8000& 4.3/0.09 & 4.3/0.09  & 7.1/0.19 &
6.1/0.12 & 3.7/\textbf{0.06} & \textbf{3.0}/0.17\\\cline{2-8}

$\lfloor\frac{2n}{100}\rfloor$ &14000&11.7/0.20 &
11.7/0.20  & 18.4/0.35 & 16.0/0.28 & 8.6/\textbf{0.08 } & \textbf{6.5}/0.12\\\cline{2-8}
 &20000&22.9/0.62 & 22.9/0.64  & 34.6/0.99 & 30.3/1.14 & 16.5/0.84 & \textbf{12.8}/\textbf{0.34}\\ \cline{1-8}
\end{tabular}
\caption{Results of the average and the standard variance of the
runtime.}\label{CS3}
\end{table}

\begin{table}
\centering
\begin{tabular}{|c|c|c|c|}
\hline
\multicolumn{1}{|c|}{$s$}&\multicolumn{1}{|c|}{$n$}
&\multicolumn{1}{|c|}{ Average
iteration/Average total NCGf}&\multicolumn{1}{|c|}{ $\frac{\mbox{Average
iteration}}{\mbox{Average total NCGf}}$}\\
\hline
&8000&16.2/21.2&0.7642\\

 $\lfloor\frac{n}{100}\rfloor$&14000& 17.8/22.0&0.8091\\

&20000&20.0/24.0& 0.8333 \\  \cline{1-4}

&8000&17.2/22.2&0.7748
\\
 $\lfloor\frac{2n}{100}\rfloor$&14000&18.8/23.0&0.8174\\

 &20000&28.2/33.6 &0.8392\\
\cline{1-4}
\end{tabular}
\caption{Results of the average total NCGf.  NCGf represents the number of computation of $\nabla f$.}\label{e5}
\end{table}

\subsection{Logistic Regression}

Given a set of training data $(x_i,y_i),i=1,
 \cdots,N$, where the input $x_i\in \mathbb{R}^n$, and the output $y_i\in\{1,-1\}$. We wish to
find a classffication rule from the training data, so that when given a new input $x$, we can
assign a class $y$ from $\{1,-1\}$ to it. For this example, we consider the following  sparse logistic regression model using $\ell_0$ regularization
\[
\min_{(u,v)\in X} \frac{1}{N}\sum_{i=1}^N\log(1+\exp(-y_i(u^Tx_i+v)))+\lambda
\|u\|_0
\]
where $X=[-10^{10},10^{10}]^{n+1}$.
The data set gisette used for our numerical experiment is taken from \cite{Edmunds2016TMAC}. The train set  contains 6000 samples of 5000 dimensions, and the test set  contains 1000 samples of 5000 dimensions.

For all the methods, the stopping criteria is commonly set to be
\[
\|x_k-x_{k-1}\|_{\infty}<5\times10^{-4},
\]
 and the initial point is obtained by FISTA\cite{BT09} for $\ell_1$ minimization (where the initial point is $zeros(n+1,1)$, and the stopping criteria is $\|x_k-x_{k-1}\|_{\infty}<0.02$, the regularization  parameter $\lambda=0.001$).
All the parameters are chosen according to accuracy. In detail, we choose the penalty parameter $\lambda=0.00005$; choose $\mu=10^{-6}$, $\omega_k=0.99$ for Algorithm 1 and EPIHT method; choose $\beta_k=10^{-6}$, $\alpha_k=(0.999999-2\beta_k)/L$ for IFB method; choose $\alpha_x=\alpha_y=1/(L+10^{-6})$ for mAPG and nmAPG method, moreover choose $\eta=0.6$ for nmAPG method.

\begin{table}
\centering
\begin{tabular}{|c|c|c|c|}
\hline
Method&Iteration number&Runtime&Accuracy\\
\hline
PIHT&443&493.7&0.9710\\
IFB&443&494.3&0.9710\\
mAPG&254&735.0&0.9700\\
nmAPG&199&301.3&0.9720\\
EPIHT&116&211.0&0.9780\\
Our&136&216.6&0.9760\\
\hline
\end{tabular}
\caption{Numerical results of logistic regression. }\label{LSR2}
\end{table}
%
%        1471        1471        3104
%
%  Columns 4 through 6
%
%        4514        1040        2436
The results are listed in Table \ref{LSR2}.  We can see that the results of Algorithm 1 and EPIHT are better than other three methods in the sense of iterations number, runtime and accuracy. Although the iteration number of Algorithm 1 is bigger than EPIHT's, their runtime is close, and the accuracy is comparable.

\section{Conclusions and perspectives}
In this paper, we proposed one proximal iterative hard thresholding type method--Algorithm 1, for solving the $\ell_0$ regularized problem. We provide some convergence analysis for the proposed method. We further show in some numerical experiments that, the algorithm 1 is faster than PIHT, IFB, mAPG, nmAPG and EPIHT or comparable with EPIHT.

\section*{Acknowledgements}

This work was partially supported by NSFC (No. 9133010 2),
the Young Top-notch Talent program of China, 973 program (No. 2015CB856004).
%\section*{References}


\begin{thebibliography}{10}

\bibitem{Attouch2008Proximal}
{\sc H.~Attouch, J.~Bolte, P.~Redont, and A.~Soubeyran}, {\em Proximal
  alternating minimization and projection methods for nonconvex problems: An
  approach based on the kurdyka-?ojasiewicz inequality}, Mathematics of
  Operations Research, 35 (2008), pp.~438--457.

\bibitem{HJ}
{\sc H.~Attouch, J.~Bolte, and B.~F. Svaiter}, {\em Convergence of descent
  methods for semi-algebraic and tame problems: proximal algorithms,
  forward--backward splitting, and regularized gauss--seidel methods},
  Mathematical Programming, 137 (2013), pp.~91--129.

\bibitem{Attouch2015The}
{\sc H.~Attouch and J.~Peypouquet}, {\em The rate of convergence of
  {N}esterov's accelerated forward-backward method is actually $o(k^{-2})$},
  Mathematics,  (2015).

\bibitem{BDHSZZ2016}
{\sc C.~Bao, B.~Dong, L.~Hou, Z.~Shen, X.~Zhang, and X.~Zhang}, {\em Image
  restoration by minimizing zero norm of wavelet frame coefficients}, 32
  (2016), p.~115004.

\bibitem{BT09}
{\sc A.~Beck and M.~Teboulle}, {\em A fast iterative shrinkage-thresholding
  algorithm for linear inverse problems}, SIAM Journal on Imaging Sciences, 2
  (2009), pp.~183--202.

\bibitem{BD08}
{\sc T.~Blumensath and M.~E. Davies}, {\em Iterative thresholding for sparse
  approximations}, Journal of Fourier Analysis and Applications, 14 (2008),
  pp.~629--654.

\bibitem{BD09}
{\sc T.~Blumensath and M.~E. Davies}, {\em Iterative hard thresholding for
  compressed sensing}, Applied and Computational Harmonic Analysis, 27 (2009),
  pp.~265--274.

\bibitem{Bolte2007Clarke}
{\sc J.~Bolte, A.~Daniilidis, A.~Lewis, and M.~Shiota}, {\em Clarke
  subgradients of stratifiable functions}, {SIAM} Journal on Optimization, 18
  (2007), pp.~556--572.

\bibitem{Bolte2008Characterizations}
{\sc J.~Bolte, A.~Daniilidis, O.~Ley, and L.~Mazet}, {\em Characterizations of
  lojasiewicz inequalities and applications}, Mathematics,  (2008).

\bibitem{JS}
{\sc J.~Bolte, S.~Sabach, and M.~Teboulle}, {\em Proximal alternating
  linearized minimization for nonconvex and nonsmooth problems}, Mathematical
  Programming, 146 (2014), pp.~459--494.

\bibitem{bot2014inertial}
{\sc R.~I. Bot, E.~R. Csetnek, and S.~C. L\'{a}szl\'{o}}, {\em An inertial
  forward¨cbackward algorithm for the minimization of the sum of two nonconvex
  functions}, EURO Journal on Computational Optimization, 4 (2016), pp.~3--25.

\bibitem{EJ06}
{\sc E.~J. Cand{\`e}s, J.~Romberg, and T.~Tao}, {\em Robust uncertainty
  principles: Exact signal reconstruction from highly incomplete frequency
  information}, Information Theory, IEEE Transactions on, 52 (2006),
  pp.~489--509.

\bibitem{CCSS03}
{\sc R.~Chan, T.~Chan, L.~Shen, and Z.~Shen}, {\em Wavelet algorithms for
  high-resolution image reconstruction}, SIAM Journal on Scientific Computing,
  24 (2003), pp.~1408--1432.

\bibitem{CW05}
{\sc P.~L. Combettes and V.~R. Wajs}, {\em Signal recovery by proximal
  forward-backward splitting}, Multiscale Modeling and Simulation, 4 (2005),
  pp.~1168--1200.

\bibitem{DZ13}
{\sc B.~Dong and Y.~Zhang}, {\em An efficient algorithm for $l_0$ minimization
  in wavelet frame based image restoration}, Journal of Scientific Computing,
  54 (2013), pp.~350--368.

\bibitem{Edmunds2016TMAC}
{\sc B.~Edmunds, Z.~Peng, and W.~Yin}, {\em T{MAC}: A toolbox of modern
  async-parallel, coordinate, splitting, and stochastic methods},  (2016).

\bibitem{KURDYKA1994W}
{\sc K.~Kurdyka and A.~Parusinski}, {\em $w_f$-stratification of subanalytic
  functions and the {\l}ojasiewicz inequality}, Comptes Rendus De Lacad\'{e}mie
  Des Sciences S\'{e}rie Math\'{e}matique, 318 (1994), pp.~129--133.

\bibitem{chubulai}
{\sc H.~Li and Z.~Lin}, {\em Accelerated proximal gradient methods for
  nonconvex programming}, In Advances in Neural Information Processing Systems
  (NIPS) 28,  (2015).

\bibitem{LM79}
{\sc P.-L. Lions and B.~Mercier}, {\em Splitting algorithms for the sum of two
  nonlinear operators}, SIAM Journal on Numerical Analysis, 16 (1979),
  pp.~964--979.

\bibitem{1963Une}
{\sc S.~{\L}ojasiewicz}, {\em Une propri\'{e}t\'{e} topologique des
  sous-ensembles analytiques r\'{e}els.}, Les \'{E}quations Aux
  D\'{e}riv\'{e}es Partielles,  (1963), pp.~87--89.

\bibitem{Lu12}
{\sc Z.~Lu}, {\em Iterative hard thresholding methods for $l_0$ regularized
  convex cone programming}, Mathematical Programming, 147 (2014), pp.~125--154.

\bibitem{Ochs2014iPiano}
{\sc P.~Ochs, Y.~Chen, T.~Brox, and T.~Pock}, {\em ipiano: Inertial proximal
  algorithm for non-convex optimization}, SIAM Journal on Imaging Sciences, 7
  (2014), pp.~1388--1419.

\bibitem{Passty1979Ergodic}
{\sc G.~B. Passty}, {\em Ergodic convergence to a zero of the sum of monotone
  operators in hilbert space}, Journal of Mathematical Analysis $\&$
  Applications, 72 (1979), pp.~383--390.

\bibitem{Pock2017Inertial}
{\sc T.~Pock and S.~Sabach}, {\em Inertial proximal alternating linearized
  minimization (i{PALM}) for nonconvex and nonsmooth problems}, 9 (2017),
  pp.~1756--1787.

\bibitem{Salzo2012Inexact}
{\sc S.~Salzo and S.~Villa}, {\em Inexact and accelerated proximal point
  algorithms}, Journal of Convex Analysis, 19 (2012), pp.~1167--1192.

\bibitem{STY-2011}
{\sc Z.~Shen, K.~C. Toh, and S.~Yun}, {\em An accelerated proximal gradient
  algorithm for frame-based image restoration via the balanced approach}, SIAM
  Journal on Imaging Sciences, 4 (2011), pp.~573--596.

\bibitem{Xu2013A}
{\sc Y.~Xu and W.~Yin}, {\em A block coordinate descent method for regularized
  multiconvex optimization with applications to nonnegative tensor
  factorization and completion}, {SIAM} Journal on Imaging Sciences, 6 (2013),
  pp.~1758--1789.

\bibitem{ZLC12}
{\sc X.~Zhang, Y.~Lu, and T.~Chan}, {\em A novel sparsity reconstruction method
  from poisson data for 3{D} bioluminescence tomography}, Journal of Scientific
  Computing, 50 (2012), pp.~519--535.

\bibitem{Zhang2015A}
{\sc X.~Zhang and X.~Q. Zhang}, {\em A note on the complexity of proximal
  iterative hard thresholding algorithm}, Journal of the Operations Research
  Society of China, 3 (2015), pp.~459--473.

\bibitem{ZDL-2013}
{\sc Y.~Zhang, B.~Dong, and Z.~Lu}, {\em $\ell_0$ minimization for wavelet
  frame based image restoration}, Mathematics of Computation, 82 (2013),
  pp.~995--1015.

\end{thebibliography}
\end{document}